\documentclass[11pt,letterpaper,reqno]{amsart}

\usepackage{amsmath}
\usepackage{amsfonts}
\usepackage{amssymb}
\usepackage{amsthm}
\usepackage{amscd}

\usepackage{latexsym}
\usepackage{mathrsfs}
\usepackage{psfrag}
\usepackage[dvips]{epsfig}
\usepackage{epsfig}

\swapnumbers 

\theoremstyle{plain}                              
\newtheorem{thm}{Theorem}[section]
\newtheorem{prop}[thm]{Proposition}
\newtheorem{defn}[thm]{Definition}
\newtheorem{lem}[thm]{Lemma}
\newtheorem{cor}[thm]{Corollary}

\newtheorem*{thmA}{Theorem A}               
\newtheorem*{thmB}{Theorem B}               
\newtheorem*{thmC}{Theorem C}               

\theoremstyle{definition}                         
\newtheorem*{example}{Example}
\newtheorem*{remark}{Remark} 

\theoremstyle{remark}                             

\numberwithin{equation}{section}

\newcommand{\R}{\mathbb{R}}                     
\newcommand{\N}{\mathbb{N}}                     

\newcommand{\ms}[1]{\mathscr{#1}}               

\newcommand{\veps}{\varepsilon}

\DeclareMathOperator{\diam}{diam}

\DeclareMathOperator{\Ker}{Ker}

\providecommand{\norm}[1]{\left\lVert#1\right\rVert}       

\providecommand{\abs}[1]{\left\lvert#1\right\rvert}        

\providecommand{\del}{\partial}

\begin{document}

\title[H\"older forms]{H\"older forms and integrability of invariant
  distributions}


\author[S. N. Simi\'c]{Slobodan N. Simi\'c}

\address{Department of Mathematics, San Jos\'e State University, San
  Jos\'e, CA 95192-0103}



\email{simic@math.sjsu.edu}


\date{\today}
\dedicatory{}



\begin{abstract}

  We prove an inequality for H\"older continuous differential forms on
  compact manifolds in which the integral of the form over the
  boundary of a sufficiently small, smoothly immersed disk is bounded
  by a certain multiplicative convex combination of the volume of the
  disk and the area of its boundary. This inequality has natural
  applications in dynamical systems, where H\"older continuity is
  ubiquitous. We give two such applications. In the first one, we
  prove a criterion for the existence of global cross sections to
  Anosov flows in terms of their expansion-contraction rates. The
  second application provides an analogous criterion for
  non-accessibility of partially hyperbolic diffeomorphisms.

\end{abstract}

\maketitle

\section{Introduction}
\label{sec:intro}

H\"older continuity is ubiquitous in dynamical systems. H\"older
continuous differential (though not differentiable) forms consequently
play an important role there, especially in hyperbolic and partially
hyperbolic dynamics. For instance, integrability of various invariant
distributions (by which we mean bundles or plane fields) and the
holonomy of the corresponding foliations can be expressed in terms of
differential forms. Anosov used differential forms extensively for
this purpose in his seminal paper \cite{anosov+67}.

Assume, for example, that $TM = E \oplus F$ is a H\"older continuous
invariant splitting for a diffeomorphism $f : M \to M$. It is often
important to know whether $E$ is an integrable distribution. One can
locally define $k = \dim F$ independent H\"older 1-forms $\alpha_1,
\ldots, \alpha_k$ such that the intersection of their kernels equals
$E$.  If $F$ admits a global frame, these forms can be defined
globally. Let $\alpha = \alpha_1 \wedge \cdots \wedge \alpha_k$. Then
$i_v \alpha = 0$, for every vector $v$ tangent to $E$ (where $i_v$
denotes the inner multiplication by $v$), so we can write
$\Ker(\alpha) = E$. The Frobenius integrability condition (see, e.g.,
\cite{warner+83}) requires that $d\alpha$ be divisible by $\alpha$,
i.e., that $d\alpha = \alpha \wedge \omega$, for some 1-form
$\omega$. Recall that this is equivalent to $d\alpha_i \wedge \alpha =
0$, for all $1 \leq i \leq k$. Since $\alpha$ is only H\"older, this
condition clearly does not apply. Hartman~\cite{hart02} (see also
Plante~\cite{plante72}) proved an analogous integrability condition
for continuous forms using the notion of the Stokes
differential. Namely, $\alpha$ is said to be Stokes differentiable if
there exists a locally integrable $(k+1)$-form $\beta$ such that
\begin{displaymath}
  \int_{\del D} \alpha = \int_D \beta,
\end{displaymath}
for every smoothly immersed $(k+1)$-disk with piecewise smooth
boundary $\del D$. The form $\beta$ is then called the Stokes
differential of $\alpha$. Hartman showed that $E$ is integrable if and
only if $\alpha$ divides $\beta$ in the above sense. The utility of
this result is limited since there are no good criteria for the Stokes
differentiability of continuous H\"older forms.

In certain dynamical situations, however, in order to show
integrability of an invariant distribution one needs less than the
Stokes or Frobenius theorem, as we will demonstrated in this
paper. The crucial inequality is given in Theorem A: for any compact
manifold $M$, there exist numbers $K, \sigma > 0$ depending on $M$,
$k$, and $\theta \in (0,1)$, such that for every $C^\theta$ H\"older
$k$-form ($1 \leq k \leq n-1$) $\alpha$ on $M$ and every $C^1$
immersed disk $D$ with piecewise $C^1$ boundary satisfying $\max\{
\diam(\del D), \abs{\del D} \} < \sigma$,
\begin{displaymath}
  \abs{ \int_{\del D} \alpha} \leq K \norm{\alpha}_{C^\theta} 
  \abs{\del D}^{1-\theta} \abs{D}^\theta.
\end{displaymath}
The idea behind the proof is simple: we approximate $\alpha$ locally
by smooth forms $\alpha^\veps$, such that $\norm{\alpha -
  \alpha^\veps}_{C^0} \leq \norm{\alpha}_{C^\theta} \veps^\theta$ and
$\norm{d\alpha^\veps}_{C^0} \leq C \norm{\alpha}_{C^\theta}
\veps^{\theta-1}$, where $C > 0$ is a universal constant. This is done
by using the standard technique of regularization. By subtracting and
adding $\alpha^\veps$ from and to $\alpha$ in $\int_{\del D} \alpha$,
it is easy to show that
\begin{equation} \label{eq:A} \abs{\int_{\del D} \alpha} \leq
  C(\theta) \norm{\alpha}_{C^\theta} (\abs{\del D} \veps^\theta +
  \abs{D} \veps^{\theta-1}),
\end{equation}
where $C(\theta)$ is a constant depending only on $\theta$. The trick
is to allow $\veps$ to vary over a sufficiently large interval and
then find the best $\veps$ by minimizing the right-hand side of
\eqref{eq:A}. For this, $\abs{D}/\abs{\del D}$ needs to be
sufficiently small. To eliminate the smallness requirement on
$\abs{D}/\abs{\del D}$, we use a special case of the isoperimetric
inequality on Riemannian manifolds, supplied by
Gromov~\cite{gromov+83}.

We give two applications of this inequality in dynamical
systems. First, we prove a criterion for the existence of a global
cross section to an Anosov flow in terms of its expansion and
contraction rates (Theorem B). We then translate this result into the
language of partially hyperbolic diffeomorphisms and give a criterion
for \emph{non}-accessibility, also in terms of expansion-contraction
rates (Theorem C). Both applications have strong limitations in that
they apply only to a ``small'' set of systems. This is not surprising,
since ``most'' distributions are not integrable. However, Theorem C
suggests that there is a certain trade-off between the size of the
H\"older exponent $\theta$ of the invariant splitting and
accessibility: if $\theta$ is better than the standard lower estimate,
accessibility is lost. This is illustrated by an example, due to an
anonymous referee.

\subsection*{Notation} If $S$ is a $k$-dimensional immersed
submanifold of a Riemannian manifold $M$, $\abs{S}$ will always denote
its Riemannian $k$-volume. If $f : M \to N$ is a smooth map between
smooth manifolds, $T_p f$ will denote its derivative (or tangent map)
$T_p M \to T_{f(p)} N$. For non-negative functions $u, v$, we will
write $u \lesssim v$ if there exists a uniform constant $c > 0$ such
that $u \leq c v$. If $u \lesssim v$ and $v \lesssim u$, we write $u
\asymp v$.

If $(X, d)$ is a metric space and $0 < \theta < 1$, recall that a
function $f : X \to \R$ is called $C^\theta$ H\"older (or just
$C^\theta$ for short) if
\begin{displaymath}
  H_\theta(f) = \sup_{x \neq y} \frac{\abs{f(x) - f(y)}}{d(x,y)^\theta} < \infty.
\end{displaymath}
The $C^\theta$-norm of $f$ is defined by
\begin{displaymath}
  \norm{f}_{C^\theta} = \sup_{x \in X} \abs{f(x)} + H_\theta(f).
\end{displaymath}


\subsection*{Acknowledgments} We thank the anonymous referees for
constructive criticism and the example at the end of the paper. While
the paper was being revised for publication, Jenny Harrison pointed
out that there are connections between Theorem A and some of her work
in \cite{harrison+98}.

\section{Preliminaries}
\label{sec:prelim}

We start with a short overview of regularization of functions and the
isoperimetric inequality.


\subsection{Regularization in $\R^n$}                          \label{sbs:reg}

We briefly review a well-known method of approximating locally
integrable functions by smooth ones, which will be used in the proof
of the main inequality.

Suppose $u: \R^n \to \R$ is locally integrable and define its
\textsf{regularization} (or mollification) as the convolution
\begin{displaymath}
  u^\veps = \eta_\veps \ast u,
\end{displaymath}
where $\eta_\veps(x) = \veps^{-n} \eta\left(\frac{x}{\veps}\right)$,
$\veps > 0$, and $\eta: \R^n \to \R$ is the \textsf{standard
  mollifier}~\cite{evans+98,stein+70}
  \begin{equation*}
    \eta(x) = \begin{cases}
      A \exp \left( \frac{1}{\abs{x}^2 - 1} \right) & 
      \text{if $\abs{x} < 1$} \\
      0 & \text{if $\abs{x} \geq 1$},
      \end{cases}
  \end{equation*}
  with $A$ chosen so that $\int \eta \: dx = 1$. Note that the support
  of $\eta_\veps$ is contained in the ball of radius $\veps$ centered
  at $0$ and $\int \eta_\veps \, dx = 1$.
\begin{prop}  \label{prop:reg}  
  
  Let $u: \R^n \to \R$ be locally integrable. Then:
\begin{enumerate}

  \item[(a)] $u^\veps \in C^\infty(\R^n)$.
  \item[(b)] If $u \in L^\infty$, then
    $\norm{u^\veps}_{L^\infty} \leq \norm{u}_{L^\infty}$.
  \item[(c)] If $u \in C^\theta$ $(0 < \theta \leq 1)$, then
    $\norm{u^\veps - u}_{C^0} \leq \norm{u}_{C^\theta} \:
    \veps^\theta$.
  \item[(d)] If $u \in C^\theta$, then 
    \begin{displaymath}
      \norm{du^\veps}_{C^0} \leq
      \norm{d\eta}_{L^1} \norm{u}_{C^\theta} \veps^{\theta-1}
    \end{displaymath}
    where $\norm{d\eta}_{L^1} = \max_i \int_{\R^n} \abs{\del \eta/\del
      x_i} dx$.
\end{enumerate}

\end{prop}
\begin{proof}
  
  Proof of (a) and (b) can be found in \cite{evans+98}. For (c), we
  have
  \begin{align*}
    \abs{u^\veps(x) - u(x)} & = \left\lvert \int_{B(0,\veps)} 
      \eta_\veps(y) [u(x-y) - u(x)] \, dy \right\rvert \\
    & \leq \norm{u}_{C^\theta} \veps^\theta \int_{B(0,\veps)}
    \eta_\veps(y) \, dy \\
      & = \norm{u}_{C^\theta} \veps^\theta.
  \end{align*}
  If $u \in C^1$, then the same estimates hold with $\theta$ replaced
  by $1$. 
  
  Observe that since $\eta_\veps$ has compact support,
  \begin{equation}                    \label{eq:compact}
    \int_{\R^n} \frac{\del \eta_\veps}{\del x_i}(y) \, dy = 0,
  \end{equation}
  for $1 \leq i \leq n$. Note also that

  \begin{equation*}
    \frac{\del \eta_\veps}{\del x_i}(x) = \frac{1}{\veps^{n+1}} 
    \frac{\del \eta}{\del x_i}\left( \frac{x}{\veps} \right).
  \end{equation*}
  Assuming $u \in C^\theta$, we obtain (d):
  \begin{align*}
    \left\lvert \frac{\del u^\veps}{\del x_i}(x) \right\rvert & =
    \left\lvert \int_{\R^n} u(x-y) \frac{\del \eta_\veps}{\del x_i}(y)
      \, dy \right\rvert \\
    & \overset{\text{by} \ \eqref{eq:compact}}{=} \left\lvert
      \int_{B(0,\veps)} [u(x-y) - u(x)] \frac{\del \eta_\veps}{\del
        x_i}(y) \, dy \right\rvert,  \\
    & \leq \norm{u}_{C^\theta} \, \veps^\theta \int_{B(0,\veps)}
      \left\lvert \frac{\del \eta_\veps}{\del x_i}(y) \right\rvert \, dy  \\
      & = \norm{u}_{C^\theta} \, \veps^\theta \int_{B(0,\veps)}
      \frac{1}{\veps^{n+1}} \left\lvert \frac{\del \eta}{\del
          x_i}\left(\frac{y}{\veps}\right)
      \right\rvert \, dy  \\
      & \overset{z = \frac{y}{\veps}}{=} \norm{u}_{C^\theta} \,
      \veps^\theta \cdot \frac{1}{\veps} \int_{B(0,1)} \left\lvert
        \frac{\del \eta}{\del x_i}(z) \right\rvert \, dz,  \\
    & \leq \norm{d\eta}_{L^1} \norm{u}_{C^\theta} \, \veps^{\theta-1}. \qedhere
  \end{align*}
\end{proof}

If $\alpha$ is a $k$-form on an open set $U \subset \R^n$, then
$\alpha$ can be regularized component-wise. Write
\begin{displaymath}
  \alpha = \sum_I a_I dx_I,
\end{displaymath}
where $I = (i_1,\ldots,i_k)$, $i_1 < \cdots < i_k$, $i_1, \ldots, i_k
\in \{1, \ldots, n\}$, and $dx_I = dx_{i_1} \wedge \cdots \wedge
dx_{i_k}$. By definition, $\alpha$ is of class $C^r$ ($r \in
[0,\infty]$) if and only if each $a_I$ is of class $C^r$. Define the
$\veps$-regularization of $\alpha$ by
\begin{equation}        \label{eq:reg}
    \alpha^\veps = \sum_I a_I^\veps dx_I,
\end{equation}
where for each $I$, $a_I^\veps$ is the $\veps$-regularization of
$a_I$.

\subsection{Regularization on Riemannian manifolds}
\label{sec:reg-manif}

Suppose now that $M$ is a compact $C^\infty$ Riemannian manifold and
fix a finite atlas $\ms{A} = \{ (U,\varphi) \}$ of $M$. Let $\alpha$
be a $C^\theta$ $k$-form on $M$. This means that $\alpha$ is of class
$C^\theta$ in every local coordinate system on $M$, i.e.,
$(\varphi^{-1})^\ast \alpha$ is $C^\theta$ on $\varphi(U)$, for each
chart $(U,\varphi) \in \ms{A}$. Define the $C^\theta$-norm of $\alpha$
by
\begin{equation}    \label{eq:alpha-norm}
  \norm{\alpha}_{C^\theta} = \max_{(U,\varphi) \in \ms{A}}
  \norm{(\varphi^{-1})^\ast \alpha}_{C^\theta(\varphi(U))}.
\end{equation}
We regularize $\alpha$ in each coordinate chart as follows. For each
$(U,\varphi) \in \ms{A}$, choose an open set $\hat{U}$ in $M$ such
that the closure of $\hat{U}$ is contained in $U$ and the collection
$\{ \hat{U} \}$ still covers $M$. Since $M$ is compact, without loss
of generality we can assume that $\varphi(U)$ -- hence
$\varphi(\hat{U})$ -- is bounded. Define
\begin{equation}  \label{eq:epsM}
  \veps_M = \min_{(U,\varphi) \in \ms{A}} 
  \inf \{ d(x,y) : x \in \del \varphi(U), 
  y \in \del \varphi(\hat{U}) \}.
\end{equation}
If the representation of $\alpha$ in the $(U,\varphi)$-coordinates is
\begin{displaymath}
 \tilde{\alpha}_U =  (\varphi^{-1})^\ast \alpha = \sum_I a_I dx_I,
\end{displaymath}
then we define the $\veps$-regularization of $\alpha$ on $\hat{U}$ by
$\alpha^\veps_U = \varphi^\ast(\tilde{\alpha}_U^\veps)$, where
$\tilde{\alpha}_U^\veps$ is the $\veps$-regularization of
$\tilde{\alpha}_U$ defined as in \eqref{eq:reg}. If $\veps \in
(0,\veps_M)$, then $\alpha_U^\veps$ is defined on $\hat{U}$, for every
$(U,\varphi) \in \ms{A}$.

This produces a family $\{ \alpha_U^\veps \}$ of $C^\infty$ $k$-forms,
with $\alpha_U^\veps$ approximating $\alpha$ on $\hat{U}$ in the sense
of Proposition~\ref{prop:reg}. Using partitions of unity, this family
can be patched together into a globally defined smooth form; however,
for our purposes local regularization will be sufficient.



\subsection{Remarks on the isoperimetric inequality}

We will also need a special case of the isoperimetric inequality on
Riemannian manifolds, which we now briefly review.

Recall that the isoperimetric inequality for $\R^n$ states (cf.,
\cite{osserman+78,gromov+07}) that for an arbitrary domain $D$ in
$\R^n$, its $n$-volume $\abs{D}$ and the $(n-1)$-volume $\abs{\del D}$
of its boundary are related as
\begin{equation}              \label{eq:iso}
  \abs{D} \leq C_n \abs{\del D}^{\frac{n}{n-1}},
\end{equation}
where 
\begin{displaymath}
  C_n = \frac{1}{(n^n \omega_n)^{1/(n-1)}},
\end{displaymath}
and $\omega_n$ denotes the volume of the unit ball in $\R^n$.

On Riemannian manifolds the situation is more complicated, so we will
only discuss a special case we need in this paper.

Given a cycle $Z$ ($\del Z = 0$) in a Riemannian manifold $M$, recall
that the \textsf{isoperimetric problem} asks whether there is a volume
minimizing chain $Y$ in $M$, such that $\del Y = Z$. For our purposes
it suffices to consider this problem for \emph{small} $Z$. The
solution is given by the following result.

\begin{lem}[\cite{gromov+83}, {\bf Sublemma 3.4.B'}]  \label{lem:gromov}
  For every compact manifold $M$, there exists a small positive
  constant $\delta_M$ such that every $k$-dimensional cycle $Z$ in $M$
  of volume less than $\delta_M$ bounds a chain $Y$ in $M$, which is
  small in the following sense:
  \begin{itemize}
  \item[(i)] $\abs{Y} \leq c_M \abs{Z}^{(k+1)/k}$, for some constant
    $c_M$ depending only on $M$;
  \item[(ii)] The chain $Y$ is contained in the $\varrho$-neighborhood
    of $Z$, where $\varrho \leq c_M \abs{Z}^{1/k}$.
  \end{itemize}
\end{lem}

The following corollary is immediate.

\begin{cor}    \label{cor:gromov}
  If $D$ is a $C^1$-immersed $(k+1)$-dimensional disk with piecewise
  $C^1$ boundary in a compact manifold $M$ with $\abs{\del D} <
  \delta_M$, then there exists a $(k+1)$-disk $\tilde{D} \subset M$
  such that $\del \tilde{D} = \del D$,
  \begin{equation}     \label{eq:tilde}
    \abs{\tilde{D}} \leq c_M \abs{\del D}^{\frac{k+1}{k}},
  \end{equation}
  and $\tilde{D}$ is contained in the $\varrho$-neighborhood of $\del
  D$, where $\varrho \leq c_M \abs{\del D}^{1/k}$.

\end{cor}

\section{The Main Inequality}
\label{sec:inequality}

We now have a necessary set-up for proving our main inequality.

\begin{thmA}      \label{thm:main}
  
  Let $M$ be a compact manifold and let $\alpha$ be a $C^\theta$
  $k$-form on $M$, for some $0 < \theta < 1$ and $1 \leq k \leq
  n-1$. There exist constants $\sigma, K > 0$, depending only on $M$,
  $\theta$, and $k$, such that for every $C^1$-immersed $(k+1)$-disk
  $D$ in $M$ with piecewise $C^1$ boundary satisfying $\max\{
  \diam(\del D), \abs{\del D} \} < \sigma$, we have
  \begin{displaymath}
    \abs{ \int_{\del D} \alpha} \leq K \norm{\alpha}_{C^\theta} 
    \abs{\del D}^{1-\theta} \abs{D}^\theta.
  \end{displaymath}
\end{thmA}

\begin{proof}
  The proof is divided into three steps. First, we show that the
  inequality holds for small, sufficiently flat disks in $\R^n$. By
  sufficiently flat, we mean that the ratio $\abs{D}/\abs{\del D}$ is
  small enough. Second, we extend this result to compact
  manifolds. Finally, we use the isoperimetric inequality to remove
  the smallness assumption on $\abs{D}/\abs{\del D}$ and complete the
  proof of the theorem.

  \subsection*{Step I} We now prove the inequality for small,
  sufficiently flat disks in $\R^n$. Let $U, \hat{U}$ be bounded
  (open) domains in $\R^n$ such that the closure of $\hat{U}$ is
  contained in $U$. Define
  \begin{displaymath}
    \hat{\veps} = \inf \{ d(x,y) : x \in \del U, y \in \del \hat{U} \}.
  \end{displaymath}
  Observe that $0 < \hat{\veps} < \infty$. Let $\alpha$ be a
  $C^\theta$ $k$-form defined on $U$. Then we have:

  \begin{prop}     \label{prop:flat}
    For every $C^1$-immersed $(k+1)$-disk $D$ in $\hat{U}$ with
    piecewise $C^1$ boundary, satisfying
    \begin{displaymath}
      \frac{\abs{D}}{\abs{\del D}} < \frac{\theta \hat{\veps}}{1-\theta},
    \end{displaymath}
      we have 
  \begin{displaymath}
    \abs{ \int_{\del D} \alpha} \leq C(\theta) \norm{\alpha}_{C^\theta}
    \abs{\del D}^{1-\theta} \abs{D}^\theta,
  \end{displaymath}
  where
  \begin{displaymath}
    C(\theta) =  \max\{ 1, \norm{d\eta}_{L^1} \} 
    \left\{ \left( \frac{1-\theta}{\theta} \right)^\theta +
      \left( \frac{\theta}{1-\theta} \right)^{1-\theta} \right\}.
  \end{displaymath}

  \end{prop}

  Here $\eta$ is as in \S\ref{sbs:reg}.

  \begin{proof}
    Let $\alpha^\veps$ be the $\veps$-regularization of $\alpha$ as
    above. If $0 < \veps < \hat{\veps}$, then $\alpha^\veps$ is
    defined on $\hat{U}$. Furthermore, by Proposition~\ref{prop:reg}
  \begin{displaymath}
    \norm{\alpha - \alpha^\veps}_{C^0} \leq \norm{\alpha}_{C^\theta} \veps^\theta 
    \qquad \text{and} \qquad \norm{d\alpha^\veps}_{C^0} \leq
    \norm{d\eta}_{L^1} \norm{\alpha}_{C^\theta} \veps^{\theta-1}.    
  \end{displaymath}
  Let $D$ be a $(k+1)$-disk in $\hat{U}$ satisfying $\abs{D}/\abs{\del
    D} < \theta \hat{\veps}/(1-\theta)$. Subtracting and adding
  $\alpha^\veps$, and using the Stokes theorem, we obtain:
  \begin{align*}
    \abs{ \int_{\del D} \alpha} & \leq  
      \abs{ \int_{\del D} (\alpha - \alpha^\veps)}
      +  \abs{ \int_{\del D} \alpha^\veps} \\
      & = \abs{ \int_{\del D} (\alpha - \alpha^\veps)} + 
      \abs{ \int_D d\alpha^\veps} \\
      & \leq \abs{\del D} \norm{\alpha}_{C^\theta} \veps^\theta +
      \abs{D} \norm{d\eta}_{L^1} \norm{\alpha}_{C^\theta} \veps^{\theta-1} \\
      & \leq \max \{ 1, \norm{d\eta}_{L^1} \}
      \norm{\alpha}_{C^\theta} (\abs{\del D} 
      \veps^\theta + \abs{D} \veps^{\theta-1}).
  \end{align*}
  The estimate is valid for all $\veps$ for which $\alpha^\veps$ is
  defined on $\hat{U}$, that is, for $0 < \veps < \hat{\veps}$. The
  minimum of
  \begin{displaymath}
    \veps \mapsto \abs{\del D} \veps^\theta + \abs{D} \veps^{\theta-1}
  \end{displaymath}
  is achieved at $\veps_\ast = (1-\theta)\abs{D}/(\theta \abs{\del
    D})$, which lies in the permissible range $(0,\hat{\veps})$. This
  minimum equals
  \begin{displaymath}
    \left\{ \left( \frac{1-\theta}{\theta} \right)^\theta +
      \left( \frac{\theta}{1-\theta} \right)^{1-\theta} \right\} 
    \abs{\del D}^{1-\theta} \abs{D}^\theta. \qedhere
  \end{displaymath}
  \end{proof}
  
\begin{remark}
  If $\alpha$ is $C^1$, then it is $C^\theta$, for all $0 < \theta <
  1$, and it is not hard to check that as $\theta \to 1-$,
  \begin{displaymath}
    C(\theta) \to \max \{ 1, \norm{d\eta}_{L^1} \}.
  \end{displaymath}
\end{remark}

\subsection*{Step II} Let $M$ be a compact Riemannian $C^\infty$
manifold. We fix an atlas $\ms{A} = \{(U,\varphi)\}$ such that each
$\varphi(U)$ is bounded. For each chart $(U,\varphi) \in \ms{A}$,
choose an open set $\hat{U} \subset U$ so that:

\begin{itemize}

\item the closure of $\hat{U}$ is contained in $U$;

\item the collection $\{ \hat{U} \}$ covers $M$.

\end{itemize}

Let $\veps_M$ be defined as in \eqref{eq:epsM} and denote the Lebesgue
number of the covering $\{ \hat{U} \}$ by $L$. This means that for
every set $S \subset M$, if $\diam(S) < L$, then $S \subset \hat{U}$,
for some chart $U$.

Define also
\begin{displaymath}
  b_- = \min_{(U,\varphi) \in \ms{A}} \inf_{p \in \hat{U}} \norm{T_p \varphi} \quad
  \text{and} \quad 
   b_+ = \max_{(U,\varphi) \in \ms{A}} \sup_{p \in \hat{U}} \norm{T_p \varphi}.
\end{displaymath}
Since $\ms{A}$ is finite and the sets $\hat{U}$ are relatively
compact, $b_-$ and $b_+$ are finite and positive.

\begin{prop}      \label{prop:small}
  
  If $\alpha$ is a $C^\theta$ $k$-form on $M$, then for every
  $C^1$-immersed $(k+1)$-disk $D$ with piecewise $C^1$ boundary in $M$
  satisfying $\diam(D) < L$ and 
  \begin{displaymath}
    \frac{\abs{D}}{\abs{\del D}} < 
    \frac{b_-^k \theta \veps_M}{(1-\theta) b_+^{k+1}},
  \end{displaymath}
  we have
  \begin{displaymath}
    \abs{ \int_{\del D} \alpha} \leq \varkappa C(\theta) 
    \norm{\alpha}_{C^\theta}
    \abs{\del D}^{1-\theta} \abs{D}^\theta,
  \end{displaymath}
  where $C(\theta)$ is the same as above, $\norm{\alpha}_{C^\theta}$
  was defined in \eqref{eq:alpha-norm}, and $\varkappa$ is
  a constant depending only on $M$, $\theta$, and $k$.
\end{prop}

  \begin{proof}
    Let $D$ be a disk satisfying the above assumptions. Since
    $\diam(D) < L$, there exists a chart $U$ such that $D \subset
    \hat{U}$. Observe that
    \begin{displaymath}
      \frac{\abs{\varphi(D)}}{\abs{\del \varphi(D)}} \leq 
      \frac{b_+^{k+1} \abs{D}}{b_-^k \abs{\del D}} < 
      \frac{\theta \veps_M}{1-\theta}.
    \end{displaymath}
    Therefore, we can use the change of variables formula and apply
    Proposition~\ref{prop:flat} to $(\varphi^{-1})^\ast \alpha$ on
    $\varphi(D)$. We obtain:
    \begin{align*}
      \abs{ \int_{\del D} \alpha} & = \abs{ \int_{\del \varphi(D)}
        (\varphi^{-1})^\ast \alpha } \\
      & \leq C(\theta) \norm{\alpha}_{C^\theta} \abs{\del \varphi(D)}^{1-\theta}
      \abs{\varphi(D)}^\theta \\
      & \leq C(\theta) \norm{\alpha}_{C^\theta} 
      (b_+^k \abs{\del D})^{1-\theta} (b_+^{k+1}
      \abs{D})^\theta \\
      & = C(\theta) \norm{\alpha}_{C^\theta} b_+^{k+\theta} \abs{\del D}^{1-\theta}
      \abs{D}^\theta.
    \end{align*}
    The completes the proof of the proposition with $\varkappa =
    b_+^{k+\theta}$.
  \end{proof}

  \subsection*{Step III} To extend the inequality to all small disks,
  we proceed as follows. Let
  \begin{displaymath}
    \sigma = \min \left\{ 1, \delta_M, \left( 
      \frac{b_-^k \theta \veps_M}{(1-\theta) b_+^{k+1} c_M} \right)^k,
       \left( \frac{L}{1+c_M} \right)^k \right\}.
  \end{displaymath}
  Here $\delta_M$ and $c_M$ are the same as in
  Lemma~\ref{lem:gromov}. Suppose that $D$ satisfies $\max \{
  \diam(\del D), \abs{\del D} \} < \sigma$. Then by
  Corollary~\ref{cor:gromov}, there exists a $(k+1)$-disk $\tilde{D}$
  such that $\del \tilde{D} = \del D$,
  \begin{displaymath}
    \abs{\tilde{D}} \leq c_M \abs{\del \tilde{D}}^{(k+1)/k},
  \end{displaymath}
  and $\tilde{D}$ is contained in the $\varrho$-neighborhood of $\del
  D$, with $\varrho \leq c_M \abs{\del D}^{1/k} < c_M
  \sigma^{1/k}$. If $\abs{D} \leq \abs{\tilde{D}}$, we can simply take
  $\tilde{D} = D$, so without loss we assume $\abs{D} >
  \abs{\tilde{D}}$.

  The above assumptions imply
  \begin{displaymath}
    \frac{\abs{\tilde{D}}}{\abs{\del \tilde{D}}} \leq c_M \abs{\del D}^{1/k} <
    \frac{b_-^k \theta \veps_M}{(1-\theta) b_+^{k+1}},
  \end{displaymath}
  and 
  \begin{align*}
    \diam(\tilde{D}) & \leq \diam(\del D) + \varrho \\
      & \leq \sigma + c_M \sigma^{1/k} \\
      & \leq \sigma^{1/k}(1+c_M) \\
      & < L
  \end{align*}
  so we can apply Proposition~\ref{prop:small} to $\alpha$ on
  $\tilde{D}$. This yields
  \begin{align*}
    \abs{ \int_{\del D} \alpha} & = \abs{ \int_{\del \tilde{D}}
      \alpha} \\
    & \leq \varkappa C(\theta) \norm{\alpha}_{C^\theta} 
    \abs{\del \tilde{D}}^{1-\theta}
    \abs{\tilde{D}}^\theta \\
    & = \varkappa C(\theta) \norm{\alpha}_{C^\theta}  \abs{\del D}^{1-\theta}
    \abs{\tilde{D}}^\theta \\
    & < \varkappa C(\theta) \norm{\alpha}_{C^\theta}  \abs{\del D}^{1-\theta}
    \abs{D}^\theta. \qedhere
  \end{align*}
\end{proof}

\begin{remark}
  \begin{itemize}
  \item[(a)] If $k = 1$, then $\diam(\del D) \leq \abs{\del D}$, so
    the assumption $\diam(\del D) < \sigma$ is superfluous.

  \item[(b)] The estimate also holds for ``long, thin'' disks $D$,
    namely, those that can be decomposed into finitely many small
    disks $D_1, \ldots, D_N$ such that $\abs{D_i} \lesssim \abs{D}/N$,
    $\abs{\del D} \lesssim \abs{\del D}/N$ and Theorem A applies to
    each $D_i$. For then $\del D = \del D_1 + \cdots + \del D_N$ and
    \begin{align*}
      \abs{\int_{\del D} \alpha} & = \abs{\sum_{i=1}^N \int_{\del D_i} \alpha} \\
      & \leq \sum_{i=1}^N K \norm{\alpha}_{C^\theta} \abs{\del
        D_i}^{1-\theta}
      \abs{D_i}^\theta \\
      & \lesssim N K \norm{\alpha}_{C^\theta} \left( 
        \frac{\abs{\del D}}{N} \right)^{1-\theta}
      \left( \frac{\abs{D}}{N} \right)^\theta \\
      & = K \norm{\alpha}_{C^\theta} \abs{\del D}^{1-\theta}
      \abs{D}^\theta.
    \end{align*}

  \item[(c)] Theorem A is also valid for immersed \emph{submanifolds}
    $D$ with piecewise smooth boundary. The proof goes through word
    for word.

  \end{itemize}

\end{remark}


\section{Global cross sections to Anosov flows}
\label{sec:anosov}

Recall that a non-singular smooth flow $\Phi = \{ f_t \}$ on a closed
(compact and without boundary) Riemannian manifold $M$ is called
\textsf{Anosov} if there exists a $Tf_t$-invariant continuous
splitting of the tangent bundle,
\begin{equation*}
  TM = E^{ss} \oplus E^c \oplus E^{uu},
\end{equation*}
and constants $C > 0$, $0 < \nu < 1$, and $\lambda > 1$ such that for
all $t \geq 0$, 
\begin{equation*}
\norm{Tf_t \! \restriction_{E^{ss}}} \leq C \nu^t 
\qquad \qquad \text{and} \qquad
\qquad \norm{Tf_t \! \restriction_{E^{uu}}} \geq C \lambda^t.
\end{equation*}
The center bundle $E^c$ is one dimensional and generated by the vector
field $X$ tangent to the flow. The distributions $E^{uu}, E^{ss},
E^{cu} = E^c \oplus E^{uu}$, and $E^{cs} = E^c \oplus E^{ss}$ are
called the strong unstable, strong stable, center unstable, and center
stable bundles, respectively. Typically they are only H\"older
continuous~\cite{hps77, hassel+97}, yet uniquely integrable
\cite{anosov+67}, giving rise to continuous foliations denoted by
$W^{uu}, W^{ss}, W^{cu}$, and $W^{cs}$, respectively. Recall that a
distribution $E$ is called \textsf{uniquely integrable} if it is
tangent to a foliation and every differentiable curve everywhere
tangent to $E$ is wholly contained in a leaf of the foliation.

The idea of studying the dynamics of a flow by introducing a (local or
global) cross section dates back to Poincar\'e. Recall that a smooth
compact codimension one submanifold $\Sigma$ of $M$ is called a
\textsf{global cross section} for a flow if it intersects every orbit
transversely. If this is the case, then every point $p \in \Sigma$
returns to $\Sigma$, defining the Poincar\'e or first-return map $g:
\Sigma \to \Sigma$. The flow can then be reconstructed by
\textsf{suspending} $g$ under the roof function equal to the
first-return time~\cite{fried+82,katok+95,schwartz+57}.

The Poincar\'e map of a global cross section to an Anosov flow is
automatically an Anosov diffeomorphism. Therefore, any classification
of Anosov diffeomorphisms immediately translates into a classification
of the corresponding class of suspension Anosov flows.

Geometric criteria for the existence of global cross sections to
Anosov flows were obtained by Plante~\cite{plante72}, who showed that
the flow admits a smooth global cross section if the distribution
$E^{su} = E^{ss} \oplus E^{uu}$ is (uniquely) integrable. He also
showed that $E^{su}$ is integrable if and only if the foliations
$W^{ss}$ and $W^{uu}$ are \textsf{jointly integrable}. This means that
in every joint foliation chart for $W^{cs}$ and $W^{uu}$, the
$W^{uu}$-holonomy takes $W^{ss}$-plaques to $W^{ss}$-plaques. Joint
integrability of $W^{uu}$ and $W^{ss}$ (in that order) is defined
analogously; by symmetry, $E^{su}$ is uniquely integrable if and only
if $W^{uu}$ and $W^{ss}$ are jointly integrable.

We now present a criterion for the existence of a global cross section
to an Anosov flow in terms of its expansion-contraction rates.

Let $p$ and $q$ be in the same local strong unstable manifold of an
Anosov flow $\Phi = \{ f_t \}$. Assume $p$ and $q$ are close enough so
that they lie in a foliation chart for both $W^{cu}$ and $W^{ss}$.

\begin{defn}     \label{def:us-disk}
  If $E^{ss}$ is of class $C^1$, a $C^1$ immersed 2-disk $D \subset M$
  is called a $us$-\textsf{disk} if:
  \begin{itemize}

  \item Its boundary is the concatenation of four simple $C^1$ paths:
    $\del D = \gamma_1 + \gamma_2 + \gamma_3 + \gamma_4$, where
    $\gamma_1 \subset W_{\mathrm{loc}}^{uu}(p)$, $\gamma_3 \subset
    W_{\mathrm{loc}}^{cu}(q)$, for some $p, q$ as above; furthermore,
    $\gamma_2 \subset W_{\mathrm{loc}}^{ss}(x)$ and $\gamma_4 \subset
    W_{\mathrm{loc}}^{ss}(p)$, where $x$ is the terminal point of
    $\gamma_1$. 

  \item $D$ is a union of $W^{ss}$-arcs, i.e., arcs contained in the
    strong stable plaques.

  \end{itemize} 

  We will call $\gamma_1$ the \textsf{base} of $D$ and $\gamma_2,
  \gamma_3, \gamma_4$ its \textsf{sides}. See Fig.~\ref{fig:disc}.

\end{defn}

\begin{figure}[htbp]
\centerline{
        \psfrag{p}[][]{$p$}
        \psfrag{q}[][]{$q$}
        \psfrag{x}[][]{$x$}
        \psfrag{Wp}[][]{$W^{cu}_\mathrm{loc}(p)$}
        \psfrag{Wq}[][]{$W^{cu}_\mathrm{loc}(q)$}
        \psfrag{Wup}[][]{$\gamma_4 \subset W^{ss}_{\text{loc}}(p)$}
        \psfrag{Wux}[][]{$\gamma_2 \subset W^{ss}_{\text{loc}}(x)$}
        \psfrag{g}[][]{$\gamma_1 \subset W^{uu}_\mathrm{loc}(p)$}
        \psfrag{D}[][]{$D$}
        \psfrag{hg}[][]{$\gamma_3$}
\includegraphics[width=0.7\hsize]{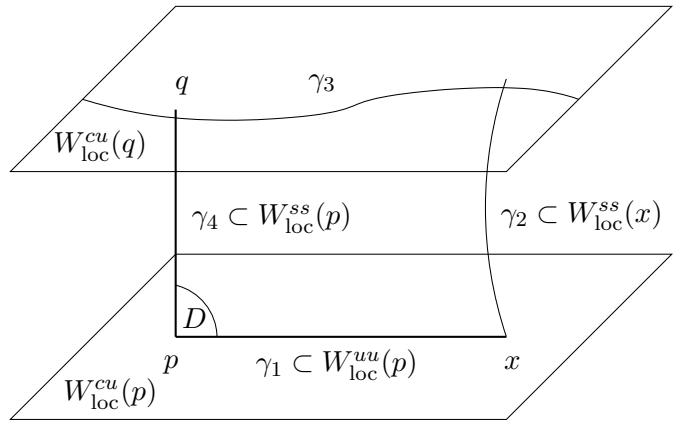}}
\caption{A $us$-disk $D$ with base $\gamma_1$.}
\label{fig:disc}
\end{figure}

  Define a 1-form $\alpha$ on $M$ by requiring
  \begin{equation}      \label{eq:alpha}
    \Ker(\alpha) = E^{ss} \oplus E^{uu}, \qquad \alpha(X) = 1.
  \end{equation}
  Since $E^{ss} \oplus E^{uu}$ is of class $C^\theta$, so is
  $\alpha$. It is clear that $\alpha$ is invariant under the flow:
  $f_t^\ast \alpha = \alpha$, for all $t \in \R$. 

\begin{lem}                                \label{lem:joint}
  
  If $E^{ss}$ is of class $C^1$, then the following statements are
  equivalent.

  \begin{enumerate}

  \item[(a)] $W^{uu}$ and $W^{ss}$ are jointly integrable.
    
  \item[(b)] $\int_{\del D} \alpha = 0$, for every $us$-disk $D$.
    
  \end{enumerate}
  
\end{lem}

\begin{proof}
  
  Follows directly from the definitions of $\alpha$ and joint
  integrability of $W^{uu}$ and $W^{ss}$.
\end{proof}

\begin{thmB}      \label{thm:unif}
  
  Suppose $\Phi = \{ f_t \}$ is a $C^2$ Anosov flow on a closed
  Riemannian manifold $M$. Assume:

  \begin{itemize}

  \item[(a)] $\norm{T f_t \! \restriction_{E^{uu}}} \leq C \mu^t$ and
    $\norm{T f_t \! \restriction_{E^{ss}}} \leq C \nu^t$, for all $t
    \geq 0$, and some constants $C > 0$, $\mu > 1$, and $0 < \nu < 1$.

  \item[(b)] $E^{ss}$ is of class $C^1$.

  \item[(c)] $\mu \nu^\theta < 1$, where $\theta \in (0,1)$ is the
    H\"older exponent of $E^{uu}$.

  \end{itemize}

  Then $\Phi$ admits a global cross section.

\end{thmB}

\begin{proof}
  Let $D$ be an $us$-disk with base $\gamma_1: [0,1] \to
  W_{\mathrm{loc}}^{uu}(p)$ and sides $\gamma_i$, $i=2,3,4$ as
  above. Then:

  \begin{lem}    \label{lem:area}

    $\abs{f_t D} \lesssim (\mu \nu)^t \abs{D}$, for all $t \geq
    0$.

  \end{lem}

  \begin{proof}
    There exists a $C^1$ vector field $Y$ tangent to $E^{ss}$ with
    flow $\{ \psi_s \}$ such that $D$ can be parametrized by
    \begin{displaymath}
      \Psi(r,s) = \psi_s(\gamma_1(r)), \qquad 0 \leq r \leq 1, \quad
      0 \leq s \leq \tau(r),
    \end{displaymath}
    for some continuous function $\tau : [0,1] \to \R$. Since $f_t
    \circ \Psi$ is a parametrization of $f_t D$, the area element of
    $f_t D$ is
    \begin{displaymath}
       \norm{Tf_t \left( \frac{\del \Psi}{\del r} \wedge \frac{\del
            \Psi}{\del s} \right)}.
    \end{displaymath}
    By the chain rule,
    \begin{displaymath}
      \frac{\del \Psi}{\del r} = T\psi_s(\dot{\gamma}_1(r)), \qquad
      \frac{\del \Psi}{\del s} = Y.
    \end{displaymath}
    The vector $T\psi_s(\dot{\gamma}_1(r))$ decomposes into $w_{ss} +
    w_c + w_{uu}$ relative to the splitting $E^{ss} \oplus E^c \oplus
    E^{uu}$. Since $\norm{Tf_t(w_{ss})} \leq C \nu^t \norm{w_{ss}} \to
    0$, as $t \to \infty$, and $\norm{Tf_t(w_c)}$ is constant, it
    follows that for $t \geq 0$,
    \begin{displaymath}
      \norm{Tf_t(T\psi_s(\dot{\gamma}_1(r)))} \lesssim \mu^t.
    \end{displaymath}
    Clearly, $\norm{Tf_t(\frac{\del \Psi}{\del s})} = \norm{Tf_t(Y)}
    \leq C \nu^t$. Therefore,
    \begin{displaymath}
      \norm{Tf_t \left( \frac{\del \Psi}{\del r} \wedge \frac{\del
            \Psi}{\del s} \right)} = \norm{Tf_t(w_{ss} \wedge Y +
        w_c \wedge Y + w_{uu} \wedge Y)},
    \end{displaymath}
    which is dominated by $\norm{Tf_t(w_{uu} \wedge Y)} \lesssim \mu^t
    \nu^t$, as $t \to \infty$. This clearly implies the claim of the
    lemma.
  \end{proof}

  By Lemma~\ref{lem:joint}, we need to show that
  \begin{displaymath}
    \int_{\del D} \alpha = 0,
  \end{displaymath}
  for every $us$-disk $D$. It is enough to prove this for small
  $D$. The idea is to use the flow invariance of $\alpha$ and change
  of variables,
  \begin{displaymath}
    \int_{\del D} \alpha = \int_{\del f_t D} \alpha,
  \end{displaymath}
  and then apply Theorem A to show that the right-hand side converges
  to zero, as $t \to \infty$. Since $f_t D$ is very ``long'', we
  cannot use Theorem A directly. However, $f_t D$ is also very
  ``thin'', since the length of its $W^{ss}$-sides go to zero, as $t
  \to \infty$. More precisely,
  \begin{displaymath}
    \abs{\del^u f_t D} \leq C \mu^t \abs{\del^u D}, \qquad
    \abs{\del^s f_t D} \leq C \nu^t \abs{\del^s D}, 
  \end{displaymath}
  where $\del^u D = \gamma_1 + \gamma_3$ and $\del^s D = \gamma_2 +
  \gamma_4$. We therefore proceed by cutting $D$ into $N$ $us$-disks
  $D_i$, as in Figure~\ref{fig:Di}. We decompose $\gamma := \gamma_1$
  as $\gamma = \gamma_1 + \cdots + \gamma_N$, so that for each $i =
  1, \ldots, N$, $\gamma_i$ gives rise to a $us$-disk $D_i$ with
  $\abs{D_i} = \abs{D}/N$. To determine how large $N$ has to be as a
  function of $t > 0$, recall that we need $\abs{\del f_t D_i} <
  \sigma$, for each $i$, in order to apply Theorem A. Using the
  notation $\del^u, \del^s$ (with clear meaning), for each $i$ we
  have:
  \begin{align*}
    \abs{\del f_t D_i} & = \abs{\del^u f_t D_i} + \abs{\del^s f_t D_i} \\
      & \leq C \mu^t \abs{\del^u D_i} + C \nu^t \abs{\del^s D_i} \\
      & \leq C C_0 \frac{\abs{\del^u D}}{N} + C \nu^t \abs{\del^s D} \\
      & \leq C_1 \abs{\del D} \left( \frac{\mu^t}{N} + \nu^t \right),
  \end{align*}
  where $C_0 > 1$ is some constant depending only on the hyperbolicity
  of the flow and $C_1 = C C_0$. So to ensure $\abs{\del f_t D_i} <
  \sigma$, we can take $N$ to be of the order $\mu^t$. More precisely,
  assuming $t$ is so large that $C_1 \abs{\del D} \nu^t < \sigma/2$,
  choose $N$ so that $C_1 \abs{\del D} \mu^t/N < \sigma/2$, i.e., 
  \begin{displaymath}
    N > \frac{2 C_1 \abs{\del D} \mu^t}{\sigma} = N_0(t).
  \end{displaymath}
  We also require that $N < 2 N_0(t)$ so that $N \asymp \mu^t$.

\begin{figure}[htbp]
\centerline{
        \psfrag{f}[][]{$f_t$}
        \psfrag{D}[][]{$D$}
        \psfrag{fD}[][]{$f_tD$}
        \psfrag{Di}[][]{$D_i$}
        \psfrag{fDi}[][]{$f_tD_i$}
        \psfrag{.}[][]{$\cdots$}
        \psfrag{g1}[][]{$\gamma_i$}
\includegraphics[width=1.0\hsize]{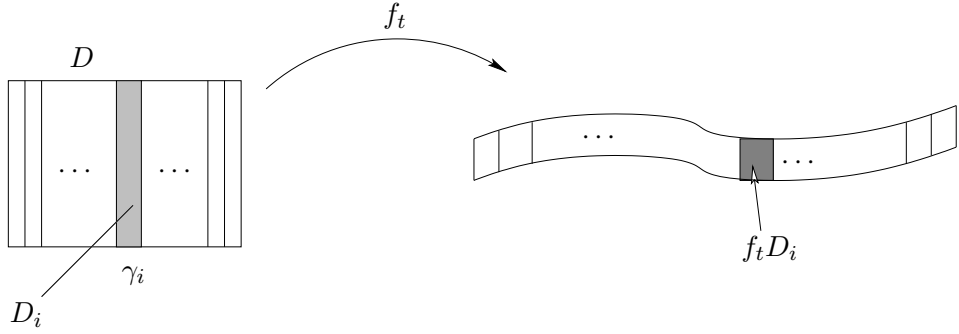}}
\caption{Decomposition of $D$ into $D_i$'s.}
\label{fig:Di}
\end{figure}

By Lemma~\ref{lem:area}, we have 
\begin{align*}
  \abs{f_t D_i} & \lesssim (\mu \nu)^t \abs{D_i} \\
       & = (\mu \nu)^t \frac{\abs{D}}{N} \\
       & \lesssim (\mu \nu)^t \frac{\abs{D}}{\mu^t} \\
       & = \nu^t \abs{D}.
\end{align*}
Applying Theorem A to each $f_t D_i$, we obtain the following
estimate:
  \begin{align*}
    \left\lvert \int_{\del D} \negmedspace \alpha \right\rvert & =
    \left\lvert \int_{\del f_tD} \negmedspace \alpha \right\rvert \\
    & \leq \sum_{i=1}^N \left\lvert \int_{\del f_t D_i} \negmedspace
      \alpha \right\rvert \\
    & \leq \sum_{i=1}^N K \norm{\alpha}_{C^\theta} \abs{\del
      f_tD_i}^{1-\theta} \abs{f_tD_i}^\theta \\
    & \lesssim N \cdot K \norm{\alpha}_{C^\theta} \sigma^{1-\theta} 
    (\nu^t \abs{D})^\theta \\
    & \lesssim K \norm{\alpha}_{C^\theta} \sigma^{1-\theta} \mu^t \nu^{\theta t} 
       \abs{D}^\theta,
    \end{align*}
    since $N \lesssim \mu^t$. Letting $t \to \infty$, we obtain
    $\int_{\del D} \alpha = 0$, as desired.
\end{proof}

\begin{remark}
  \begin{enumerate}

  \item[(a)] It is likely that Theorem B could be slightly improved by
    using the extra smoothness of $\alpha$ along the leaves of the
    center unstable foliation. This extra smoothness comes from the
    fact that along the leaves of the center unstable foliation
    $W^{cu}$, the strong unstable distribution $E^{uu}$ (assumed to be
    only H\"older) is actually as smooth as the flow, i.e., $C^2$. The
    condition $\mu \nu^\theta < 1$ in Theorem B would then be replaced
    by a weaker one $\mu \nu^\tau < 1$, where $\tau = (2-\theta)^{-1}$.

  \item[(b)] It needs to be pointed out that the assumptions (b) and
    (c) of Theorem B are quite restrictive and are satisfied only by a
    small set of Anosov flows. However, once a system $\Phi$ does
    verify (b) and (c), by structural stability there exists a $C^1$
    neighborhood $\ms{U}$ of $\Phi$ such that each flow in $\ms{U}$
    admits a global cross section.

  \end{enumerate}
\end{remark}



\section{Accessibility}
\label{sec:accessibility}

In this section we prove a sufficient condition for a partially
hyperbolic diffeomorphism to be non-accessible.

Recall that a diffeomorphism $f$ of a compact Riemannian manifold $M$
is called \textsf{partially hyperbolic} if the tangent bundle of $M$
splits continuously and invariantly into the stable, center, and
unstable bundle, $TM = E^s \oplus E^c \oplus E^u$, such that $Tf$
exponentially contracts $E^s$, exponentially expands $E^u$ and this
hyperbolic action on $E^s \oplus E^u$ dominates the action of $Tf$ on
$E^c$. The stable and unstable bundles are always uniquely integrable,
giving rise to the stable and unstable foliations, $W^s, W^u$. In
contrast, the center bundle $E^c$, the center stable $E^{cs} = E^c
\oplus E^s$, and the center unstable bundle $E^{cu} = E^c \oplus E^u$
are not always integrable. If they are, $f$ is called
\textsf{dynamically coherent} (cf., \cite{brin+03,ps+04,bw+07}).

A partially hyperbolic diffeomorphism is called \textsf{accessible} if
every two points of $M$ can be joined by an $su$-path, that is, a
continuous piecewise smooth path consisting of finitely many arcs
lying in a single leaf of $W^s$ or a single leaf of $W^u$.

If $f$ is dynamically coherent and the foliations $W^s$ and $W^u$ are
jointly integrable (in the same sense as in Section
\S\ref{sec:anosov}), then it is clear that $f$ is not accessible. We
can also speak of joint integrability of $W^u$ and $W^s$ (in that
order), which is defined analogously; it also implies non-accessibility.

Let $f : M \to M$ be a partially hyperbolic diffeomorphism with $\dim
E^c = \ell$ and integrable center-unstable bundle $E^{cu}$. Assume
that both $E^s$ and $E^u$ have dimension $\geq \ell$. We now define
objects that will play the role of $us$-disks in this context.

Assume $E^s$ is $C^1$ and pick an arbitrary $p \in M$ and $q \in
W^s_\text{loc}(p)$. Let $\Gamma$ be an $\ell$-dimensional
$C^1$-immersed surface (with piecewise $C^1$ boundary) contained in
$W^u_\text{loc}(p)$ and define $D$ to be a $C^1$-immersed
$(\ell+1)$-disk (or ``cube'') by making the following requirements:

\begin{itemize}

\item $\del D \cap W^{cu}_\text{loc}(p) = \Gamma$;

\item $D$ is foliated by arcs tangent to $E^s$;

\item $\del D \cap W^{cu}_\text{loc}(q) = h^s(\Gamma)$, where $h^s :
  W^{cu}_\text{loc}(p) \to W^{cu}_\text{loc}(q)$ is the holonomy map
  associated with the stable foliation $W^s$.

\end{itemize}

We will call any disk (or ``cube'') $D$ satisfying these requirements
a $us$-\textsf{cube} for $f$. We will refer to $\Gamma$ as the
\textsf{base} of $D$. We also write $\del^u D = \Gamma + h^s(\Gamma)$
and $\del^s D = \del D - \del^u D$.

Denote the restriction of $Tf$ to $E^\rho$ by $T^\rho f$, with $\rho
\in \{s, c, u \}$, and by
\begin{displaymath}
    m(T^c_p f) = \min \{ \norm{T_p f(v)} : v \in E^c(p), \norm{v} = 1 \}.
\end{displaymath}
the conorm (or minimum norm) of $T_p^c f$. We also define

\begin{displaymath}
  \lambda_\rho = \norm{T^\rho f} = \max \{ \norm{T_p^\rho f} : p \in M \},
\end{displaymath}
where $\rho \in \{s,u\}$, and $m(T^c f) = \min\{ m(T_p^c f) : p \in M
\}$.

\begin{thmC}
 
  Suppose that $f : M \to M$ is a $C^2$ partially hyperbolic
  diffeomorphism of a compact Riemannian manifold $M$. Assume:

  \begin{itemize}

  \item[(a)] $E^s$ is $C^1$;

  \item[(b)] $E^{cu}$ is integrable;

  \item[(c)] $E^c$ is a trivial bundle (i.e., it admits a global frame);

  \item[(d)] $\ell = \dim E^c \leq \min(\dim E^s, \dim E^u)$;

  \item[(e)] $f$ satisfies
    \begin{displaymath}
      \frac{\norm{T^u f}^\ell \norm{T^s f}^\theta}{m(T^c f)^\ell} < 1,
    \end{displaymath}
    where $\theta \in (0,1)$ is the H\"older exponent of $E^u$ and
    $E^c$.

  \end{itemize}
 
  Then $W^u$ and $W^s$ are jointly integrable, hence $f$ is not
  accessible.

\end{thmC}

\begin{proof}
  The proof is analogous to that of Theorem B. Let $\{ X_1, \ldots,
  X_\ell \}$ be a global $C^\theta$ frame for $E^c$, and define
  1-forms $\alpha_1, \ldots, \alpha_\ell$ on $M$ by requiring that
  $\alpha_i(X_i) = 1$ and $\Ker(\alpha_i) = E^s \oplus E^u \oplus \R
  X_1 \oplus \cdots \oplus \widehat{\R X_i} \oplus \cdots \oplus \R
  X_\ell$, where the hat denotes omission. Then $\alpha = \alpha_1
  \wedge \cdots \wedge \alpha_\ell$ is an $\ell$-form satisfying
  \begin{displaymath}
    \Ker(\alpha) = E^s \oplus E^u, \qquad 
    \alpha(X_1,\cdots, X_\ell) = 1.
  \end{displaymath}
  Since $E^s$ is $C^1$, and $E^c$ and $E^u$ are $C^\theta$, it follows
  that $\alpha$ is $C^\theta$. By construction,
  \begin{displaymath}
    f^\ast \alpha = (\det T^c f) \: \alpha.
  \end{displaymath}
  It is easy to see that $W^u$ and $W^s$ are jointly integrable if and
  only if
  \begin{displaymath}
  \int_{\del D} \alpha = 0,
  \end{displaymath}
  for every $us$-cube $D$. Let us prove this is indeed the case. (Note
  that our goal is not to use Hartman's version of the Frobenius
  theorem to prove integrability of $E^s \oplus E^u$.)

  Let $D$ be an arbitrary $us$-cube. We would like to imitate the
  proof of Theorem B to show that $\int_{\del D} \alpha = 0$. Let
  $\sigma > 0$ be as in Theorem A and choose $k_0 \in \N$ sufficiently
  large so that $k \geq k_0$ implies $\lambda_s^k \abs{\del D} <
  \sigma/2$. For each fixed $k \geq k_0$, divide $\Gamma$ into $N$
  $C^1$ immersed $\ell$-dimensional surfaces $\Gamma_1, \ldots,
  \Gamma_N$, so that $\Gamma = \Gamma_1 + \cdots + \Gamma_N$ and let
  $D_i$ be the $us$-cube with base $\Gamma_i$ (where we assume we
  have fixed $p$ and $q \in W^s_\text{loc}(p)$ as in the definition of
  a $us$-cube). Then $D = D_1 + \cdots + D_N$ and we can choose
  $\Gamma_i$'s so that $\abs{D_i} = \abs{D}/N$.

  We need to take $N$ large enough so that $\max (\diam (\del f^k
  D_i), \abs{\del f^k D_i}) < \sigma$, for each $i$. Since
    \begin{align*}
      \abs{\del f^k D_i} & = \abs{\del^u f^k D_i} + \abs{\del^s f^k D_i} \\
        & \lesssim \lambda_u^{k \ell} \abs{\del^u D_i} 
           + \lambda_s^k \abs{\del^s D_i} \\
        & \lesssim \lambda_u^{k \ell} \frac{\abs{\del D}}{N} + \frac{\sigma}{2},
    \end{align*}
    it suffices to take $N \asymp \lambda_u^{k \ell}$. Note that this
    ensures not only $\abs{\del f^k D_i} < \sigma$, but also that
    $\diam(\del f^k D_i) \lesssim \lambda_u^k \leq \lambda_u^{k \ell}$
    is $< \sigma$, as desired.

    An argument completely analogous to that in Lemma~\ref{lem:area}
    shows that for each $i = 1, \ldots, N$ and $k \in \N$,
  \begin{displaymath}
    \abs{f^k D_i} \lesssim (\lambda_u^\ell \lambda_s)^k \abs{D_i}.
  \end{displaymath}
  Since $N \asymp \lambda_u^{k \ell}$ and $\abs{D_i} = \abs{D}/N$, we
  have $\abs{f^k D_i} \lesssim \lambda_s^k \abs{D}$.

  Using $\abs{\det T^c f^{-k}} \leq m(T^c f)^{-k \ell}$, we obtain:
  \begin{align*}
    \abs{\int_{\del D} \alpha} & = \abs{\int_{\del f^k D} 
      (f^{-k})^\ast \alpha} \\
    & \leq \sum_{i=1}^N \abs{\int_{\del f^k D_i} (\det T^c f^{-k}) \alpha} \\
    & \leq \sum_{i=1}^N K \norm{(\det T^c f^{-k}) \alpha}_{C^\theta}
    \abs{\del f^k D_i}^{1-\theta} \abs{f^k D_i}^\theta \\
    & \lesssim N \cdot K \norm{\alpha}_{C^\theta} m(T^c f)^{-k \ell} 
    \sigma^{1-\theta} (\lambda_s^k \abs{D})^\theta \\
    & \lesssim  \lambda_u^{k \ell} \cdot K \norm{\alpha}_{C^\theta} 
       m(T^c f)^{-k \ell} \sigma^{1-\theta} \lambda_s^{\theta k} \abs{D}^\theta \\
    & = K \norm{\alpha}_{C^\theta} \sigma^{1-\theta} \left( \frac{\lambda_u^\ell
        \lambda_s^\theta}{m(T^c f)^\ell} \right)^k \abs{D}^\theta,
  \end{align*}
  which converges to zero, as $k \to \infty$. This completes the proof.
\end{proof}

\begin{remark}

  (a) The assumption that $E^c$ is trivial is only used to obtain a
  globally defined form $\alpha$.
  
  (b) We point out that condition (e) requires the H\"older exponent
  $\theta$ of $E^u$ to be better than its expected value given by (see
  \cite{psw+97})
  \begin{displaymath}
    \frac{m(T^u f)}{\norm{T^c f}} m(T^s f)^\theta > 1.
  \end{displaymath}
  While this is a clear limitation of Theorem C, it also seems to
  suggest that there is a certain trade-off between accessibility and
  the size of $\theta$. Namely, if $\theta$ exceeds its expected
  value, then accessibility is lost, as illustrated by the following
  example. We are grateful to an anonymous referee for suggesting it.

\end{remark}

\begin{example}       \label{ex}
  Let $\xi$ be the real root of the equation $p(x) = x^3 - x - 1 =
  0$. Then $\xi > 1$ and the other two roots of $p$ are complex
  conjugate numbers that lie inside the unit circle (in other words,
  $\xi$ is a Pisot number). Denote by $\eta < 1$ their common absolute
  value and let $A$ be the companion matrix for the polynomial
  $p$. Since the constant term of $p$ is $-1$, we have $\xi \eta^2 =
  \abs{\det A} = 1$. The matrix $A$ defines a linear Anosov
  diffeomorphism $f_A$ of the torus $\mathbb{T}^3$. Let $f_0 :
  \mathbb{T}^4 \to \mathbb{T}^4$ be the direct product of $f_A^{-1}$
  and the identity on $S^1$. Then $f_0$ is a partially hyperbolic
  diffeomorphism, with one dimensional center and stable bundles, and
  the unstable bundle of dimension two. Following Brin~\cite{brin+75,
    pesin04}, one can make an arbitrarily small perturbation of $f_0$
  in the $C^\infty$ topology to obtain an \emph{accessible} partially
  hyperbolic diffeomorphism $f$. In particular, $f$ does not satisfy
  the conclusion of Theorem C. The only hypothesis $f$ violates is
  (e), which means that
  \begin{displaymath}
    \frac{\norm{T^u f}^\ell \norm{T^s f}^\theta}{m(T^c f)^\ell} \geq 1,
  \end{displaymath}
  where $\ell = \dim E^c = 1$ and $\theta \in (0,1)$ is the H\"older
  exponent of $E^u$ and $E^c$. To as close an approximation as one
  wishes, we have $\norm{T^s f} = \xi^{-1}$, $\norm{T^u f} =
  \eta^{-1}$, and $m(T^c f) = 1$. Therefore, up to a small error we
  obtain $\eta^{-1} \xi^{-\theta} \geq 1$, or $\eta \xi^\theta \leq
  1$. On the other hand, the standard estimate of the H\"older
  exponent $E^u$ gives a lower bound on $\theta$: namely (again up to
  a small error), $1/(\xi^{-\theta} \eta^{-1}) < 1$, that is, $\eta
  \xi^\theta > 1$. This suggests the following conclusion: if $f_0$ is
  perturbed so as to become accessible, then the H\"older exponent of
  the splitting has to become as small as the standard lower estimate
  allows; any H\"older continuity better than the worst case forces
  joint integrability of $W^s$ and $W^u$.
\end{example}

\bibliographystyle{amsplain} 

\end{document}